\documentclass{amsart}
\usepackage{amsmath,amsfonts,amsthm,amssymb,nameref}
\usepackage{thmtools}
\usepackage{thm-restate}
\usepackage{url}

\usepackage{color}
\usepackage[nobysame]{amsrefs}

\usepackage{enumerate}
\usepackage{xfrac}
\usepackage[pdftex]{hyperref}
\usepackage[all]{hypcap}
\usepackage{microtype}

\newcommand{\N}	{\mathbb N}
\newcommand{\Z}	{\mathbb Z}

\newcommand{\Cay}	{\operatorname{Cay}}
\newcommand{\diam}	{\operatorname{diam}} 
\newcommand{\girth}	{\operatorname{girth}} 

\newcommand{\Div}	{\operatorname{Div}}

\newtheorem{thm}{Theorem}[section]
\newtheorem{prop}[thm]{Proposition}
\newtheorem{lem} [thm]{Lemma} 

\newtheorem*{thm*} {Theorem} 
\newtheorem*{prop*}{Proposition}
\newtheorem*{lem*} {Lemma} 
\newtheorem*{cor*} {Corollary}
\theoremstyle{definition}

\newtheorem{remark}[thm]{Remark}
\newtheorem*{defi*}{Definition}
\newtheorem*{example*}{Example}
\newtheorem*{remark*}{Remark}
\newtheorem*{problem*}{Problem}
\newtheorem*{convention*}{Convention}

\newcommand{\sdg}[2]{\Theta_{#1}^{(#2)}}

\begin{document}

\title{Divergence and quasi-isometry classes of random Gromov's monsters}

\author{Dominik Gruber}
\address{Department of Mathematics, ETH Zurich, 8092 Zurich, Switzerland}
 \email{dominik.gruber@math.ethz.ch}

\author{Alessandro Sisto}
\address{Department of Mathematics, ETH Zurich, 8092 Zurich, Switzerland}
 \email{sisto@math.ethz.ch}

 \begin{abstract}
 We show that Gromov's monsters arising from i.i.d.\ random labellings of expanders (that we call random Gromov's monsters) have linear divergence along a subsequence, so that in particular they do not contain Morse quasigeodesics, and they are not quasi-isometric to Gromov's monsters arising from graphical small cancellation labellings of expanders.
 
 Moreover, by further studying the divergence function, we show that there are uncountably many quasi-isometry classes of random Gromov's monsters.
\end{abstract}

\maketitle

\section{Introduction}

There are two known types of \emph{Gromov's monsters} (plus derived constructions), meaning finitely generated groups ``containing'' infinite expander graphs in their Cayley graphs (in a reasonable geometric sense). Gromov's monsters were the first groups shown to not coarsely embed into Hilbert space \cite{Gromov_random,Arzhantseva-Delzant} and, moreover, they are the only known counterexamples to the Baum-Connes conjecture with coefficients \cite{HLS}.

Groups of the first type, which we will call \emph{random Gromov's monsters}, come from a \emph{random model} of finitely generated infinitely presented groups introduced in \cite{Gromov_questions}. Roughly, the model involves choosing uniformly at random labellings on a suitable family of expander graphs. The images of the corresponding expanders in the Cayley graphs are close to being quasi-isometrically embedded (i.e.\ the additive constants go to infinity in a controlled way) \cite{Gromov_random,Arzhantseva-Delzant}, see also \cite{Coulon}. See Section~\ref{section:randommonster} for the formal setup.

The second type of Gromov's monsters are certain infinitely presented \emph{graphical small cancellation groups}. They are obtained from labellings of families of expander graphs satisfying the graphical $Gr'(1/6)$-condition. (See \cite{Gru-TAMS} for a definition of the condition.) The existence of such labellings for certain families of expander graphs has been proven in \cite{Osa-label} using a probabilistic argument. The graphical small cancellation condition was developed in \cite{Gromov_random,Ollivier}, see also \cite{Gru-TAMS},  and it ensures that the resulting Cayley graph contains {\it isometrically} embedded copies of the expander graphs.

The two types of Gromov's monsters look superficially similar, since they are both constructed by labelling a family of expander graphs in such a way that a suitable small cancellation condition is satisfied (in the case of random Gromov's monsters this is the geometric small cancellation condition \cite{Gromov_random,Arzhantseva-Delzant,Coulon}). However, as it turns out, they are very different. 

The first major difference was discovered in \cite{GST}, where it is shown that random Gromov's monsters cannot act non-elementarily on hyperbolic spaces, while infinitely presented graphical small cancellation groups are acylindrically hyperbolic \cite{GS-smallcanc}. 

In view of the major open problem whether acylindrical hyperbolicity is a quasi-isometry invariant, this result motivated the question whether random Gromov's monsters can be quasi-isometric to any infinitely presented graphical small cancellation groups. In this paper, we provide a negative answer by studying a quasi-isometry invariant of finitely generated groups called \emph{divergence} for random Gromov's monsters.

\subsection*{Divergence} Roughly speaking, the divergence function measures lengths of paths avoiding specified balls as a function of the radius of the ball (see Section~\ref{section:divergence}), and it was first studied in \cite{Gr-asinv} and \cite{Ger-div}. Our first main result is that the divergence of random Gromov's monsters is linear along a subsequence. Before stating this more precisely and discussing it, we review Gromov's construction:
\begin{enumerate}
 \item Start with $G$ a non-elementary torsion-free hyperbolic group, $p\in(0,1)$, and $(\Theta_n)_{n\in\N}$ a $d$-regular expander graph such that $\diam(\Theta_n)/\girth(\Theta_n)$ is uniformly bounded.
 \item Choose a suitably large $j$ and take $(\sdg nj)_{n\in\N}$ the $j$-edge-subdivision of $(\Theta_n)_{n\in\N}$.
 \item Choose a sufficiently sparse subsequence $\Sigma:=(\Sigma_n)_{n\in\N}$ of $(\sdg nj)_{n\in\N}$.
 \item Consider the uniform random $S$-labelling of $\Sigma$ and, with slight abuse, denote by $G/\Sigma$ the quotient of $G$ by the normal subgroup generated by all group elements given by words read along closed paths in $\Sigma$.
\end{enumerate}

We show (see Theorem \ref{thm:main}):

\begin{thm}
 In the notation above, with probability at least $p$ we have that the divergence of $G/\Sigma$ is linear on a subsequence of $\N$ equivalent to $(\girth(\Sigma_n))_{n\in\N}$. 
\end{thm}

Notice that the divergence of  infinitely presented graphical small cancellation groups is superlinear, since acylindrically hyperbolic groups contain Morse elements \cite{Si-hypembmorse}, and groups with Morse elements have superlinear divergence \cite{DMS-div}.

In particular, random Gromov's monsters and  infinitely presented graphical small cancellation groups not only cannot be isomorphic to each other, they cannot even be quasi-isometric to each other. Moreover, combining our result with results in \cite{DMS-div}, one sees that random Gromov's monsters have some asymptotic cones without cut-points, while all asymptotic cones of  infinitely presented graphical small cancellation groups have cut-points.

We note that having superlinear divergence and/or Morse elements are best thought of as hyperbolic features of a given group. We think of our result of saying that random Gromov's monsters lack such hyperbolic features: It would have been reasonable to expect them since random Gromov's monsters are limits of hyperbolic groups, and by comparison with  infinitely presented graphical small cancellation groups.

\subsection*{Quasi-isometry types} When $\Sigma$ is sufficiently sparse, the divergence function is not linear. This is due to the fact that at scales intermediate between the sizes of the expander graphs, random Gromov's monsters ``look like'' hyperbolic groups. We exploit this to show that, along a different subsequence than in Theorem \ref{thm:main}, the divergence is arbitrarily close to exponential. In turn, we use this to distinguish quasi-isometry classes of random Gromov's monsters, see Theorem \ref{thm:qi}.

\begin{thm}
 In the notation above, varying the sequence $\Sigma$ yields uncountably many quasi-isometry classes of $G/\Sigma$.
\end{thm}

The analogous result for graphical small cancellation groups arising from $Gr'(1/6)$-labellings of expander graphs was proven in \cite{Hume} using the notion of separation profile of a group. We remark that the only previously known examples of groups with divergence function which is linear along a subsequence but not linear were constructed in \cite{OOS-lacun}.

\subsection*{Outline of proof} The proof of Theorem \ref{thm:main} has three main ingredients, that we explain in simplified form here. The first one, carried out in Lemma \ref{lem:many_random_words}, is that all geodesics in (the natural Cayley graph of) $G/\Sigma$ of length much smaller than the girth of $\Sigma_n$ actually appear in $\Sigma_n$, meaning that one can find a geodesic in $\Sigma_n$ with the same label. This is roughly similar to the fact that random words of length $n$ contain any word of length much smaller than $\log n$ as a subword.

Proving linear divergence is essentially finding linear detours, and the next step is to show that detours can be found in $\Sigma_n$. This uses the expansion property, and it is carried out in Lemma \ref{lem:linear_detours_in_expanders}.

However, the combination of the first two steps is not sufficient to prove Theorem \ref{thm:main} because the expanders are not embedded in $G/\Sigma_n$ sufficiently nicely, meaning that there's a gap between the scale at which every geodesic can be represented and the scale at which the embedding of $\Sigma_n$ into $G/\Sigma$ is well-behaved. Hence, one has to ``get out'' of the smaller scale in order to be able to exploit the geometry of $\Sigma_n$. This is the most sophisticated part of the paper and it uses random walks in $G/\Sigma$ arising from random labellings of geodesics in $\Sigma_n$, see Lemma \ref{lem:extendo-paths}.

\medskip

\subsection*{A remark on expanders inside random Gromov's monsters} We point the interested reader to a (negative) observation on the quasi-isometry constants for the maps taking the defining expanders into the Cayley graphs of random Gromov's monsters, namely that the additive constant cannot be sublinear in the girth, see Remark~\ref{remark:notalmostqi}.

\section{Background and notation}

\subsection{Gromov's random model}\label{section:randommonster} We recall Gromov's model for random groups obtained from i.i.d.\ labellings of (infinite sequences of) finite graphs \cite{Gromov_questions,Gromov_random}.

A generating set $S$ of a group $G$ is an epimorphism $F(S)\to G$, where $F(S)$ denotes the free group on $S$. Let $\Omega$ be a graph (in the notation of Serre \cite{Serre-trees}). We denote by $V(\Omega)$ its vertex set, by $E(\Omega)$ its edge set, and we write $|\Omega|$ for $|V(\Omega)|$. An $S$--labelling $\alpha$ of $\Omega$ is a map $E(\Omega)\to S\sqcup S^{-1}\subset F(S)$, so that $\alpha(e^{-1})=\alpha(e)^{-1}$ for every $e\in E(\Omega)$. Denote by $\mathcal A(\Omega,S)$ the set of $S$--labellings of $\Omega$. If $\Omega$ is finite, we endow it with the uniform distribution. If $\Omega$ is a disjoint union of finite connected graphs $\Omega_n$, i.e. $\Omega=\sqcup_{n\in\N}\Omega_n$, we endow $\mathcal A(\Omega,S)$ with the product distribution coming from the uniform distributions on the $\mathcal A (\Omega_n,S)$. We call this distribution the \emph{uniform random $S$-labelling of $\Omega$}. Given a sequence of events $P_n$ in a probability, we say $P_n$ holds \emph{asymptotically almost surely (a.a.s.)} if the probability that $P_n$ holds goes to $1$ as $n\to\infty$.

Given a group $G$ generated by $S$ and an $S$-labelling of a graph $\Omega$, we denote by $G/\Omega$ the quotient of $G$  by all the words labelling closed paths in $\Omega$. Notice that for each connected component $\Omega'$ of $\Omega$, there is a label-preserving graph homomorphism $\Omega'\to\Cay(G/\Omega,S)$, where $\Cay(G/\Omega,S)$ is considered with its natural $S$-labelling, and this homomorphism is unique up to choices of base points.

\subsection{Divergence}\label{section:divergence} We recall the definition of divergence of a Cayley graph \cite[Definitions 3.1 and 3.3, pp. 2496]{DMS-div}. Let $X$ be a Cayley graph of a group with respect to a finite generating set (considered as geodesic metric space). For $a,b,c\in X$, let ${\rm div}(a,b,c)={\rm div}_2(a,b,c;1/2)$ be the infimum of the lengths of paths connecting $a, b$ and avoiding the ball $B_{r/2-2}(c)$, where $r=d(c,\{a,b\})$. Define the \emph{divergence function} ${\rm Div}(n)={\rm Div}_2(n;1/2)$ as the supremum of all numbers ${\rm div}(a, b, c)$ with $d(a, b) \leq n$. Observe that, by definition, $\Div(n)\geq n$, and $\Div(n)$ is non-decreasing.

We say two maps $f_1,f_2:[0,\infty)\to[0,\infty]$ are \emph{equivalent} if there exists $L>0$ such that for $\{i,j\}=\{1,2\}$ and for all $t\in I$:
$$f_i(t)\leq Lf_j(Lt)+Lt+L,$$
and we call $L$ a \emph{comparison constant} for $f_1$ and $f_2$. If $I\subseteq [0,\infty)$ and the above inequality holds for all $t\in I$, then $f_1$ and $f_2$ are \emph{equivalent on $I$}.

Any two quasi-isometric Cayley graphs have equivalent divergence functions \cite{DMS-div}. Thus, up to equivalence, we can speak of the divergence function of a finitely generated group. 

\subsection{Cheeger constant and expander graph}
Given a subset $A$ of the vertex set of a finite graph $\Gamma$, denote by $\partial A$ the set of vertices in $V(\Gamma)\setminus A$ that can be connected to a vertex in $A$ by at least one edge. We define the Cheeger constant of $\Gamma$ as $h(\Gamma):=\min\left\{\frac{|\partial A|}{\min\{|A|,|V(\Gamma)\setminus A|\}}:A\subseteq V(\Gamma)\right\}$. Given an infinite sequence of finite graphs, we say they form an \emph{expander graph} if their vertex degrees are uniformly bounded from above, their sizes go to infinity, and their Cheeger constants are uniformly bounded away from zero. See \cite{Lubotzky} for further information.

The \emph{girth} of a graph is the infimum of all lengths of homotopically non-trivial closed paths. By \cite{Margulis, Selberg}, one example of an expander for which the ratios diameter over girth are uniformly bounded, a necessary and sufficient condition for the following results as well as for both existing constructions of Gromov's monsters, is the sequence $\bigl(\Cay(\mathrm{SL}_2(\Z/p\Z),\{A_p,B_p\}\bigr)_{p}$ where $p$ runs over all odd primes and $$A_p:=\begin{pmatrix} \overline1 & \overline2 \\ \overline0 & \overline1\end{pmatrix} \text{ and } B_p:=\begin{pmatrix} \overline1 & \overline0 \\ \overline2 & \overline1\end{pmatrix}.$$
For further examples of such graphs, see \cite{Arzhantseva-Biswas,Lubotzky}.

Given a graph $\Gamma$, we denote by $\Gamma^{(j)}$ its $j$-subdivision (for a positive integer $j$) obtained by replacing each edge by a line-graph of length $j$. We record the following fact which, in particular, implies that the $j$-subdivision of an expander is itself an expander.

\begin{lem}\label{lem:subdivision_cheeger_constant} Let $h,d>0$ and $j\in \N_{>0}$. Then there exists $h_d^{(j)}>0$ such that, if $\Gamma$ is a finite graph with $h(\Gamma)\geq h$ and vertex degree bounded above by $d$, then $h(\Gamma^{(j)})\geq h_d^{(j)}$.
\end{lem}

\begin{proof} This follows from \cite[Lemma~7.6]{Arzhantseva-Delzant} together with \cite[Propositions~4.24 and 4.25]{Lubotzky}.
\end{proof}

\subsection{Spectral radius and Kazhdan constant} Let $G$ be a group generated by a non-empty finite set $S$, and consider the Markov operator $M_S:=\frac{1}{2|S|}{\bf1}_S\in\mathbb{C}[G]$. The \emph{spectral radius of $G$ with respect to $S$} is the spectral radius of $M_S$ considered as operator $\ell^2(G)\to\ell^2(G)$ via the left-regular representation \cite{Kesten}. (This can also be rephrased in terms of the random walk on $\Cay(G,S)$. We use this in Lemma~\ref{lem:random_walk_avoids_balls}.)

The \emph{Kazhdan constant} of $G$ with respect to $S$ (in the notation of \cite{Arzhantseva-Delzant}) is the largest $0\leq \kappa\leq 1$ such that for every unitary representation $\pi$ of $G$ we have $\mathrm{spec}(M_S)\subseteq[-1,\kappa]\cup\{1\}$. A group has \emph{property~(T)} if its Kazhdan constant (with respect to some finite generating set) is $<1$.

\section{Linear divergence along the sequence of girths}

In this section, we prove our main result:

\begin{restatable}{thm}{thmmain}
\label{thm:main} Let $G$ be a non-elementary torsion-free hyperbolic group with a finite generating set $S$, let $(\Theta_n)_{n\in\N}$ be a $d$-regular expander graph with $\diam(\Theta_n)\leq C\girth(\Theta_n)$ for every $n$, for some $d,C>0$, and let $p\in(0,1)$. Then there exists $\gamma>0$ such that for every $\eta>0$ there exists 
$j_0>0$ such that for every integer $j\geq j_0$ there exists a subsequence $\Sigma:=(\Sigma_n)_{n\in\N}$ of $(\sdg nj)_{n\in\N}$ such that 
for the uniform random $S$-labelling of $\Sigma$, with probability at least $p$ we have that for every subsequence $\Omega:=(\Omega_n)_{n\in\N}$ of $\Sigma$:

\begin{itemize}
 \item(linear divergence) the divergence of $G/\Omega$ is equivalent to a linear map on a subsequence of $\N$ equivalent to $(\girth(\Omega_n))_{n\in\N}$ and
 \item(embedded expanders) for every $n$, every label-preserving map $f_n:\Omega_n\to\Cay(G/\Omega,S)$ and every $x,y\in V(\Omega_n)$ we have $$d(f_n(x),f_n(y))\geq \gamma\cdot (d(f_n(x),f_n(y))-\eta\girth(\Theta_n)).$$
\end{itemize}
In fact, $\gamma$ only depends on $G,S,C$, and $j_0$ only depends on $G,S,C,\eta,p,h$, where $h>0$ is a lower bound for the Cheeger constants of $(\Theta_n)_{n\in\N}$.
\end{restatable}

Our contribution is the first conclusion. The second conclusion is the result of \cite{Gromov_random,Arzhantseva-Delzant} and, whenever we choose $\eta<1/2$, gives a weak embedding and hence implies failure of the Baum-Connes conjecture with coefficients \cite{HLS} and of coarse embeddability into Hilbert space \cite{Matousek-exp,Gromov_questions,Gromov_random}.

\begin{remark}\label{remark:constants}
If $\Sigma'$ is a subsequence of $\Sigma$, then the uniform random $S$-labelling of $\Sigma'$ equals the distribution obtained from the random $S$-labelling of $\Sigma$ by restricting the maps $E(\Sigma)\to S\sqcup S^{-1}$ to the subset $E(\Sigma')$. Clearly any subsequence of $\Sigma'$ is a subsequence of $\Sigma$. We deduce that the properties of $\Sigma$ go to every subsequence $\Sigma'$ of $\Sigma$.

We remark that the dependency on $G$ and $S$ is, in fact, only a dependency on an upper bound $k$ for $|S|$ and an upper bound $\kappa<1$ for the Kazhdan constant of a non-elementary torsion-free hyperbolic property~(T) quotient $H$ of $G$.

We also remark that the comparison constants involved in the statement are independent of the particular subsequence of graphs.

This follows from the way Proposition~\ref{prop:main} and Lemma~\ref{lem:div_paths} are applied in the proof of Theorem~\ref{thm:main}.
\end{remark}

We will deduce Theorem~\ref{lem:div_paths} from the following proposition and from known results in the construction of random Gromov's monsters, see Proposition~\ref{prop:graphical_sc}.
The proposition says that in $\Cay(G/\sdg nj)$, there exist detours of linear length at scale roughly $\girth (\sdg nj)$. Moreover,  the subdivision parameter $j$ and the scales involved only depend on the quality of the random walk on $G$ (captured by the size of the generating set $k$ and the bound on the spectral radius $\kappa$) and certain combinatorial and metric properties (degree $d$, Cheeger constant at most $h$, ratio diameter over girth at most $C$) of the expander $(\Theta_n)_{n\in\N}$.

\begin{restatable}{prop}{propmain}
\label{prop:main}
 Let $k>0, \kappa<1, C>0, d>0, h>0$. Then there exists $j_0>0$ such that for every integer $j\geq j_0$ there exist $\epsilon_0,\nu_0,L_0>0$ such that the following holds. Let $G$ be a non-elementary torsion-free hyperbolic group with a finite generating set $S$ of size at most $k$, such that the spectral radius of $G$ w.r.t.\ $S$ is at most $\kappa$, and let  $(\Theta_n)_{n\in\N}$ be a $d$-regular expander with $\diam(\Theta_n)\leq C\girth(\Theta_n)$ and $h(\Theta_n)\geq h$ for every $n$. Then asymptotically almost surely, we have the following for the uniform random labelling of $\sdg nj$ by $S$. Let $m,x_1,x_2$ be vertices of $\Cay(G/\sdg nj)$ satisfying $\epsilon_0\girth\left(\sdg nj\right)\leq d(m,x_1)\leq d(m,x_2)\leq 2\epsilon_0\girth\left(\sdg nj\right)$. Then there exists a path of length at most $L_0\girth\left(\sdg nj\right)$ that connects $x_1$ to $x_2$ and does not intersect the $\nu_0\girth\left(\sdg nj\right)$-ball around $m$.
\end{restatable}

Proposition~\ref{prop:main} only considers certain triples of points. We will use the following lemma to find detours for all relevant triples and thus control the divergence function.

\begin{lem}\label{lem:div_paths}
Let $X$ be the Cayley graph of an infinite group with respect to a finite generating set. Suppose that there exist $1/4\geq\epsilon>0,L\geq 1,R\in\N$ with $\epsilon R\geq1$ so that for any vertices $x_1,x_2,m$ that satisfy $d(x_i,m)=R$, there exists a path from $x_1$ to $x_2$ of length at most $LR$ that avoids $B_{\epsilon R}(m)$. Then $\Div(\epsilon R)\leq (L+4)R+1$.
\end{lem}

\begin{proof}
Let $a',b',c'\in X$ with $d(a',b')\leq \epsilon R$, and let $c'\in X$. Replace $a',b',c'$ with vertices of $X$ that lie within distance $1/2$ of $a,b,c$. We now construct a path $\alpha$ from $a$ to $b$ avoiding $B(c,r/2-1)$, where $r=d(c,\{a,b\})$, and the length of the path will be at most $(L+4)R$. This gives a path of length at most $(L+4)R+1$ from $a'$ to $b'$ avoiding $B(c',d(c',\{a',b'\})/2-2)$.

First, if $r\geq 2\epsilon R$, we can just let $\alpha$ be a geodesic from $a$ to $b$, which has length at most $\epsilon R+1$. In fact, such a geodesic does not enter the ball of radius $r-(\epsilon R+1)\geq r/2-1$ around $c$.

If $r\leq 2\epsilon R$, then both $a$ and $b$ lie within distance $R$ of $c$, since $d(c,a)\leq r+d(a,b)\leq 3\epsilon R+1\leq R$, and similarly for $b$. We now replace them with points $\hat{a},\hat{b}$ at distance exactly $R$ from $c$. A standard argument gives that there exists a geodesic ray $\gamma_a$ starting at $a$ that avoids $B_{r/2}(c)$: Consider a bi-infinite geodesic $\beta$ through $a$, and assume that both rays of $\beta$ starting at $a$ have points $y,w$ in $B_{r/2}(c)$. Then, on one hand we have $d(y,w)< r$, and on the other $d(y,w)> 2(d(a,c)-r/2)$. Hence $2d(a,c)-r< r$, which contradicts $d(a,c)\geq r$.

Let $\hat{a}$ on $\gamma_a$ be so that $d(\hat{a}, c)=R$, and construct $\gamma_b$ and $\hat{b}$ similarly. Then by assumption there exists a path $\hat{\alpha}$ from $\hat{a}$ to $\hat{b}$ of length at most $LR$ that avoids $B_{r/2}(c)$, since $\epsilon R\geq r/2$.

The required path $\alpha$ is the concatenation of a subpath (of length at most $2R$) of $\gamma_a$, $\hat{\alpha}$, and a subpath (of length at most $2R$) of $\gamma_b$.
\end{proof}

\begin{lem}\label{lem:many_random_words} Let $k>0$. Then there exists $\epsilon>0$ such that for every set $S$ with $|S|\leq k$ and every expander $(\Gamma_n)_{n\in\N}$, asymptotically almost surely for the uniform random $S$-labelling of $\Gamma_n$, every word in $S$ of length at most $\epsilon\log|\Gamma_n|$ labels a simple path in $\Gamma_n$.
\end{lem}
By a word in $S$, we mean an element of the free monoid on $S\sqcup S^{-1}$. (In particular, we do not require that it is reduced.)

\begin{proof} By \cite[Theorem~4]{Hamiltonian}, there exists $\nu>0$ (only depending on the maximal degree and infimal Cheeger constant of $(\Gamma_n)_{n\in\N})$ such that each $\Gamma_n$ contains a simple path of length at least $\nu|\Gamma_n|$. 

It is a well-known fact that there exists $r>0$ only depending on the upper bound $k$ for $|S|$ such that the probability that a random word in $S$ of length $l$ contains every word of length at most $r\log(l)$ goes to $1$ as $l\to\infty$ (for completeness, we give the argument in Remark~\ref{remark:manywordsproof} below). We apply this to the label of the simple path of length at least $\nu|\Gamma_n|$ to get the claim for $\epsilon:=\frac r2\leq r\left(1+\frac{\log(\nu)}{\log(|\Gamma_n|)}\right)=\frac{r\log(\nu|\Gamma_n|)}{\log(|\Gamma_n|)}$, where the inequality holds if $n$ and hence $|\Gamma_n|$ is large enough.  
\end{proof}

\begin{remark}\label{remark:notalmostqi}
Observe that Lemma~\ref{lem:many_random_words} in particular applies to all freely trivial words of length at most $\epsilon\log |\Gamma_n|$. Recall that if $(\Theta_n)_{n\in\N}$ is an expander graph for which the ratios diameter over girth are uniformly bounded, then there exists some $\delta>0$ such that for every $n$ we have $\girth(\Theta_n)\geq \delta\log|\Theta_n|$ (see our proof of Proposition~\ref{prop:main} for an explanation of this).

Thus, applying the Borel-Cantelli Lemma to Lemma~\ref{lem:many_random_words} shows: Let $G$ be a group generated by a finite set $S$ and $\Theta:=(\Theta_n)_{n\in\N}$ an expander graph for which the ratios diameter over girth are uniformly bounded. Then there exists $\epsilon'>0$ such that for the uniform random $S$-labelling of $\Theta$ we have with probability 1 that there exist a subsequence $(\Theta_{k_n})_{n\in\N}$ and geodesic paths $p_{k_n}$ in $\Theta_{k_n}$ of length at least $\epsilon'\girth(\Theta_{k_n})$ such that any label-preserving map $\Theta\to \Cay(G/\Theta,S)$ maps every $p_{k_n}$ to a closed path.

In particular, with probability 1 any label-preserving map $\Theta\to \Cay(G/\Theta,S)$ is \emph{not} an almost quasi-isometric embedding (as defined in \cite[Definition~1]{Finn-Sell}), in contrast to a claim that first appeared in print in \cite{Finn-Sell}.

Our observation also shows that with probability 1 the map is not a coarse embedding, a fact which was already observed in \cite{Osa-label}. 
\end{remark}

\begin{remark}\label{remark:manywordsproof}
 Given $ r>0$, the probability that at least one word of length $\lfloor  r\log(n)\rfloor$ does not appear in a random word of length $n$ can be estimated as follows. Fix a word of length $\lfloor  r\log(n)\rfloor$, and fix a maximal number of disjoint subwords of the same length of the random word. The probability that the fixed word does not coincide with any given subword is $1-(2|S|)^{-\lfloor r\log(n)\rfloor}$. Since the subwords are disjoint, the probability that the fixed word does not coincide with any of the subwords is $(1-(2|S|)^{-\lfloor r\log(n)\rfloor})^{\bigl\lfloor\frac{n}{\lfloor r\log(n)\rfloor}\bigr\rfloor}$. Taking a union bound we get that the probability that at least one word of length $\lfloor  r\log(n)\rfloor$ does not appear in a random word of length $n$ is at most $(2|S|)^{\lfloor r\log(n)\rfloor}(1-(2|S|)^{-\lfloor r\log(n)\rfloor})^{\bigl\lfloor\frac{n}{\lfloor r\log(n)\rfloor}\bigr\rfloor}$. This quantity goes to 0 if and only if $(2|S|)^{\lfloor r\log(n)\rfloor}(1-(2|S|)^{-\lfloor r\log(n)\rfloor})^{\frac{n}{\lfloor r\log(n)\rfloor}}$ goes to 0. By applying the logarithm and then using the power series expansion of $\log(1-(2|S|)^{-\lfloor r\log(n)\rfloor})$ we observe that the probability goes to 0 if $\lfloor  r\log(n)\rfloor^2=o(n\cdot (2|S|)^{-\lfloor r\log(n)\rfloor})$. This holds if $ r<1/\log(2k)$ since $1/\log (2k)\leq 1/\log (2|S|)$.
\end{remark}

Adapting a standard technique from the theory of expander graphs, see e.g. \cite[Lemma~3.1.6]{Kowalski}, we find detours of linear length within an expander graph.

\begin{lem}\label{lem:linear_detours_in_expanders} Let $h>0$. Then there exists $L>0$ such that for every $d\in \N, d\geq 3$ there exists $g_0>0$ such that for every $d$-regular graph $\Gamma$ of girth at least $g_0$ and every $j\in\N_{>0}$ such that $h(\Gamma^{(j)})\geq h$ we have: let  $0<\lambda_1\leq 1/4$, let $m$ be a vertex in $\Gamma^{(j)}$, and let $v,w$ be vertices in $\Gamma^{(j)}$ with $\lambda_1\girth\Gamma^{(j)}\leq d(m,v),d(m,w)\leq1/4\cdot\girth\Gamma^{(j)}$. Then there exists a path from $v$ to $w$ in $\Gamma$ of length at most $L\log\left|\Gamma^{(j)}\right|$ that does not intersect $B_{\lambda_1\girth\Gamma^{(j)}}(m)$.
\end{lem} 

\begin{proof}
Assume $\Gamma$ is non-empty. Let $v$ be as above, and let $g:=\girth\Gamma$. Then $jg=\girth\Gamma^{(j)}$. Given $r>0$, we denote by $B_{r}'(v)$ the set of vertices of $\Gamma^{(j)}$ that can be reached from $v$ by a path of length less than $r$ not containing any vertices of $B_{\lambda_1jg}(m)$. Since $d(m,v)\leq \frac{jg}{4}$ and $\lambda_1\leq \frac{1}{4}$, we have that $B_{\lambda_1jg}(m)$ is a subtree of the tree $B_{\frac{jg}{2}}(v)$ that does not contain $v$. Thus it is contained in only one of the branches of $B_{\frac{jg}{2}}(v)$ at $v$. Hence, by considering the ball of radius $\frac{jg}{2}-j$ around a vertex of degree $d$ closest to $v$ among those not in $B_{\lambda_1jg}(m)$ we can bound $\left|B_{\frac{jg}{2}}'(v)\right|$ from below: There exists a(n explicit) constant $c$ depending only on $d$ such that $\left|B_{\frac{jg}{2}}'(v)\right|\geq cj (d-1)^{\frac g2}$.

Similarly, by considering the ball of radius $\lambda_1jg+j$ around a degree $d$ vertex closest to $m$, we can bound $|B_{\lambda_1jg}(m)|$ from above: There exists a(n explicit) constant $C$ depending only on $d$ such that 
$|B_{\lambda_1jg}(m)|\leq C j (d-1)^{\frac g4}$.

We deduce $|B_{\lambda_1jg}(m)|/\left|B_{\frac{jg}{2}}'(v)\right|\leq \frac h2$ if $g\geq g_0$ for some $g_0$ only depending on $d$ and $h$. 

There exists $C_h>0$ only depending on $h$ such that $\diam\Gamma^{(j)}\leq C_h\log\left|\Gamma^{(j)}\right|$. This is a standard argument analogous to the following computation, considering the sizes of growing balls around two given vertices, see e.g.\ \cite[Lemma~3.1.6]{Kowalski}. 

Assume $\left|B_{\frac{jg}{2}}'(v)\right|\leq \frac 12\left|\Gamma^{(j)}\right|$. Then $|B_{\frac{jg}{2}+1}'(v)|$ is at least the cardinality of the $1$--neighborhood of $B_{\frac{jg}{2}}'(v)$ (which can be estimated using $h(\Gamma^{(j)}$) minus the cardinality of the set $B_{\lambda_1jg}(m)$ that we want to avoid):
$$\left|B_{\frac{jg}{2}+1}'(v)\right|\geq \left(1+h(\Gamma^{(j)})\right)\left|B_{\frac{jg}{2}}'(v)\right|-|B_{\lambda_1jg}(m)|\geq \left(1+\frac h2\right)\left|B_{\frac{jg}{2}}'(v)\right|.$$

Similarly, for every integer $i\geq 0$, we have $\left|B_{\frac{jg}{2}+i}'(v)\right|\geq (1+\frac h2)^i\left|B_{\frac{jg}{2}}'(v)\right|$ as long as $|B_{\frac{jg}{2}+i-1}'(v)|\leq\frac12\left|\Gamma^{(j)}\right|$. Thus, there is

\begin{align*}
 c_1&\leq\left\lceil\log_{1+\frac h2}\left(\frac 12\frac{\left|\Gamma^{(j)}\right|}{\left|B_{\frac{jg}{2}}'(v)\right|}\right)\right\rceil+\frac{jg}{2}\leq\log\left|\Gamma^{(j)}\right|/\log\left(1+\frac h2\right)+1+\frac{jg}{2} \\
    &\leq\log\left|\Gamma^{(j)}\right|/\log\left(1+\frac h2\right) + 2\diam\Gamma^{(j)}\leq \left(1/\log\left(1+\frac h2\right)+2C_h\right)\log\left|\Gamma^{(j)}\right|
\end{align*}

for which $|B_{c_1}'(v)|\geq 1/2\left|\Gamma^{(j)}\right|$. 
We may apply the same arguments to $w$ to obtain $c_2$ for which $|B_{c_2}'(w)|\geq 1/2\left|\Gamma^{(j)}\right|$ satisfying the same inequalities as $c_1$. 

Clearly, $B_{c_1}'(v)\cap B_{c_2}'(w)\neq \emptyset$ as both are disjoint from the non-empty $B_{\lambda_1jg}(m)$. This implies there is a path from $v$ to $w$ that does not intersect $B_{\lambda_1\girth\Gamma^{(j)}}(m)$ of length at most $c_1+c_2\leq 2 (1/\log(1+\frac h2)+2C_h)\log\left|\Gamma^{(j)}\right|$. Thus, we may set $L:=2 (1/\log(1+\frac h2)+2C_h)$.
\end{proof}

As explained in the outline, we need to get away from a small scale. We will use random walks to achieve this. The following lemma says that if a random walk starts away from a given ball, it will not enter it.

\begin{lem}\label{lem:random_walk_avoids_balls} Let $k>0$, $\kappa<1$, and $\nu>1$. Then there exists an integer $\phi>0$ with the following property: let $G$ be a group generated by a finite set $S$ with $|S|\leq k$ such that the spectral radius of $G$ w.r.t. $S$ is at most $\kappa$. Let $r>0$ be an integer, and let $(w_n)$ be the simple random walk on $G$ starting at $g\in G$, where $g$ is at distance at least $(1+\phi)r$ from the identity. Then the probability that $(w_n)$ hits the ball $B$ of radius $r$ around the identity is at most $\nu^{-r}$. 
\end{lem}

\begin{proof}
 Since the distance between $g$ and $B$ is at least $\phi r$, with probability 1 we have that $w_n$ does not lie in $B$ for any for $n< \phi r$. Also, it is well-known that for any $h\in G$ and any $n$, we have that the probability that $w_n=h$ is at most $\kappa^n$, see e.g. \cite[Lemma 8.1-(b)]{Woess:rw}. Hence, the probability that $(w_n)$ enters $B$ is at most
$$\sum_{h\in B} \sum_{n\geq \phi r} \kappa^n\leq \#B \frac{\kappa^{\phi r}}{1-\kappa}.$$
 
 Since $|S|\leq k$, we have $\# B\leq (2k)^r$, so the quantity above is at most $\frac{(2k\kappa^\phi)^r}{(1-\kappa)}$. Fixing $k$, $\kappa<1$, and $\nu>1$, it is possible to choose $\phi$ to make this quantity smaller than $\nu^{-r}$ for every $r\geq 1$.
\end{proof}

The goal of the next lemma is to find paths in the expander $\sdg nj$ that get away from a small ball in $\Cay(G,S)$.
We use the fact that the random labelling of a path in $\sdg nj$ corresponds to a random walk in $G$. Due to the regularity and large girth of the graph, there are sufficiently many sufficiently disjoint (and thus sufficiently independent) paths such that (using Lemma~\ref{lem:random_walk_avoids_balls}) at least one of them will get away from a small ball in $G$. 

\begin{lem}\label{lem:extendo-paths}  Let $k>0$, $\kappa<1$, and $\phi>0$ be the constant obtained applying Lemma~\ref{lem:random_walk_avoids_balls} with $\nu=2$. Let $0<\epsilon\leq \frac 18$. Then we have: 
let $G$ be a group generated by a finite set $S$ with $|S|\leq k$ such that the spectral radius of $G$ w.r.t.\ $S$ is at most $\kappa$. 
Let $d,j,C>0$ and suppose $\left(\sdg nj\right)_{n\in\N}$ is a sequence of $j$-subdivisions of $d$-regular graphs $(\Theta_n)_{n\in\N}$ such that for each $n$ we have $\diam\left(\sdg nj\right)\leq C\girth\left(\sdg nj\right)$, and $|\sdg nj|\to\infty$ as $n\to\infty$. Then, for the uniform random labelling of $\sdg nj$ by $S$, a.a.s.\ every triple of vertices $m,v_1,v_2$ with $\epsilon \girth\left(\sdg nj\right)\leq d(m,v_1)\leq d(m,v_2)\leq \frac 14 \girth\left(\sdg nj\right)$ satisfies one of the following:
\begin{itemize}
 \item For some $i$, the label of the unique geodesic from $m$ to $v_i$ is not geodesic in $G$.
 \item For each $i$ there exists a geodesic path $q_i$ in $\sdg nj$ starting at $v_i$ and ending at a vertex at distance $\left\lfloor\frac{\girth\left(\sdg nj\right)}{2}\right\rfloor$ from $m$ such that, if $r_i$ is the unique geodesic from $m$ to $v_i$, then for  any label-preserving map $r_iq_i\to\Cay(G,S)$, the image of $q_i$ does not intersect the ball of radius $\frac \epsilon{2+2\phi}\girth\sdg nj$ around the image of $m$.
\end{itemize}
\end{lem}

\begin{proof} 
Set $g_n= \girth(\Theta_n)$.
Consider the subset $\Sigma$ of the sphere in $\sdg nj$ of radius $\lfloor \frac{\epsilon jg_n}{2}\rfloor$ around $v_1$ consisting of all points $\sigma$ so that the geodesic from $\sigma$ to $v_1$ only intersects the geodesic from $m$ to $v_1$ in $v_1$.
Then $\Sigma$ has at least $(d-1)^{\frac{\epsilon g_n}{2}-2}$ vertices. Now, since $\diam(\Theta_n)\leq C\girth(\Theta_n)$, we have $|\sdg nj|=|\Theta_n|+(j-1)\frac d2|\Theta_n|\leq \frac{jd}2|\Theta_n|\leq \frac{jd}{2}2(d-1)^{Cg_n+1}$. Thus, there exists some $\mu>0$ independent of $n$ such that, if $n$ is large enough, then $|\Sigma|\geq |\sdg nj|^\mu$.

For each $\sigma\in \Sigma$, we choose a simple path $p_\sigma$ starting at $\sigma$ and terminating at a vertex at distance $\lfloor \frac{jg_n}{2}\rfloor$ from $m$, such that any two such paths do not share edges. Moreover, we require that any such path intersects the closed ball of radius $\lfloor \frac{\epsilon jg_n}{2}\rfloor$ around $v_1$ only at its starting point. Such paths exist since $B_{\frac{jg_n}{2}}(m)$ is a tree and there are no degree 1 vertices in $\sdg nj$.

\begin{figure}[h]
 \includegraphics[width=.8\textwidth]{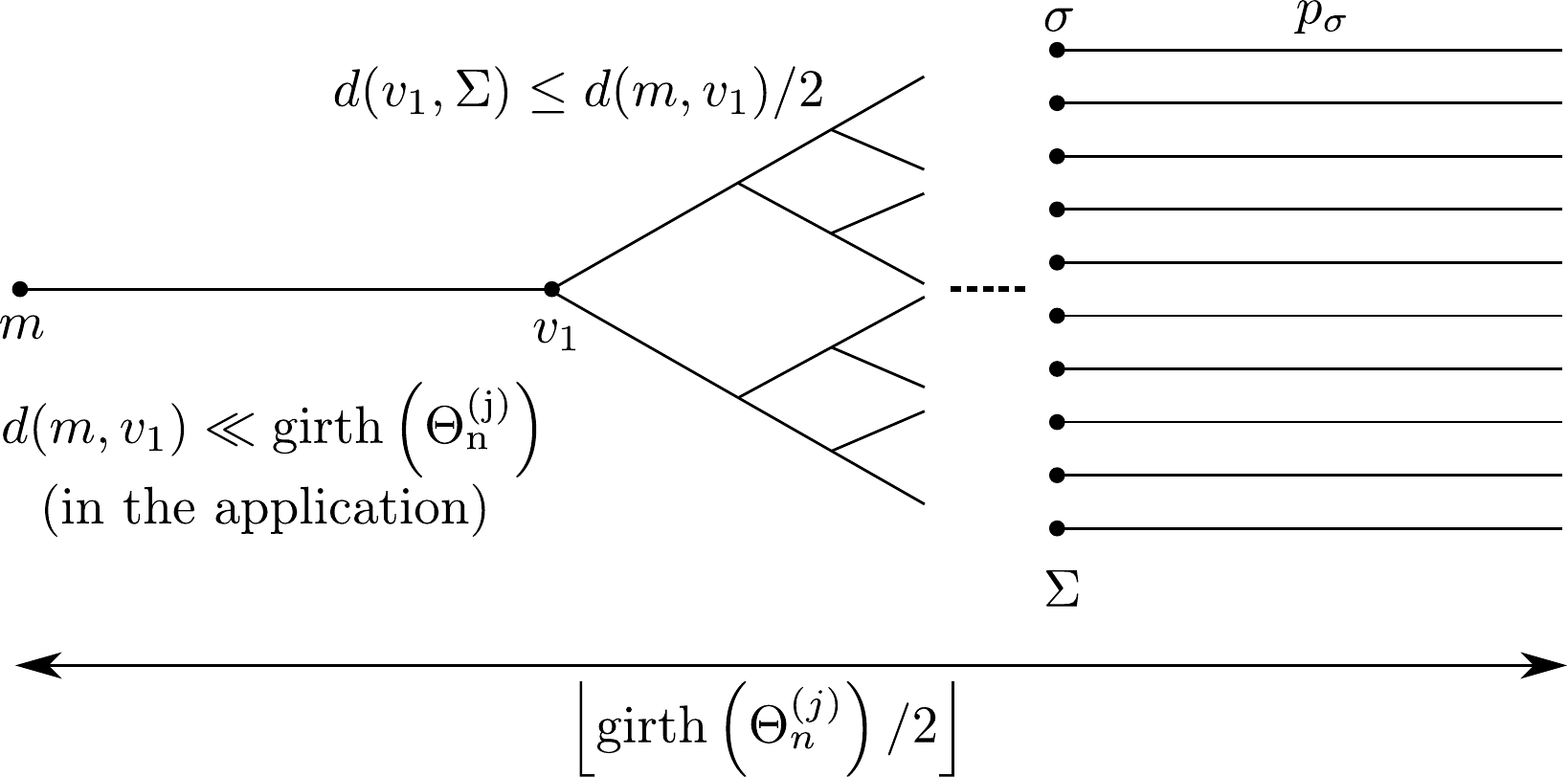}
 \caption{To get away from $m$ in $G$, we move out in $\Theta_n^{(j)}$ in many different ways and use a probabilistic argument to show that one of them does not dip into a small ball around $m$ in $G$ (while this is clear in $\Theta_n^{(j)}$). The candidate paths are paths that go to $\Sigma$, and then proceed further, so that they are disjoint after $\Sigma$.} 
\end{figure}

Let $\Gamma$ be the subgraph of $\sdg nj$ that is the union of the (unique) geodesic from $m$ to $v_1$ and the closed ball of radius $\lfloor \frac{\epsilon jg_n}{2}\rfloor$ around $v_1$. We fix a labelling of $\Gamma$ and consider conditional probabilities. Observe that the paths $p_\sigma$ do not contain edges whose label we are conditioning on, i.e.\ their labels are independent of that of $\Gamma$, so that their labels actually give us random walks. 

In case the label of the geodesic from $m$ to $v_1$ is not geodesic in $G$, our desired property is automatically satisfied. Now consider the case that the label is geodesic in $G$. For $\sigma\in\Sigma$, let $d_G(m,\sigma)$ denote the word length of the element of $G$ represented by the label of the unique geodesic from $m$ to $\sigma$. (The element is determined by the labelling of $\Gamma$.) Then, for each $\sigma\in \Sigma$, we have $d_G(m,\sigma)\geq\frac{\epsilon jg_n}{2}$ and hence $d_G(m,\sigma)\geq\lceil\frac{\epsilon jg_n}{2}\rceil$. Hence, applying Lemma~\ref{lem:random_walk_avoids_balls}, we get that for each $\sigma\in\Sigma$, the probability that the image of $p_\sigma$ enters the $\lceil\frac{\epsilon jg_n}{2+2\phi}\rceil$-ball in $G$ around the image of $m$, denoted $B_G$, is at most $2^{-\lceil\frac{\epsilon jg_n}{2+2\phi}\rceil}\leq2^{-\frac{\epsilon jg_n}{2+2\phi}}$. Since the $p_\sigma$ are disjoint, and hence their labels are independent, the probability that all the images of the $p_\sigma$ enter $B_G$ is at most $2^{-\frac{\epsilon jg_n}{2+2\phi}|\Sigma|}\leq 2^{-\frac{\epsilon jg_n}{2+2\phi}|\sdg nj|^{\mu}}$.

We do the same computation for $v_2$ and get the same probability estimate. Thus, the probability that on at least one end all extensions enter $B_G$ is at most $2\cdot 2^{-\frac{\epsilon jg_n}{2+2\phi}|\sdg nj|^{\mu}}$. 
 
We have to consider at most $|\sdg nj|^3$ triples of points. The probability that at least one of these triples fails to satisfy our desired property is then at most $|\sdg nj|^3 2\cdot2^{-\frac{\epsilon jg_n}{2+2\phi}|\sdg nj|^{\mu}}$. This goes to $0$ as $n\to\infty$.
\end{proof}

The following is the inductive step in the construction of random Gromov's monsters \cite{Gromov_random,Arzhantseva-Delzant,Coulon} as explained in \cite{Arzhantseva-Delzant}. In order to clarify the claims on constants we make, we will give the proof following \cite{Arzhantseva-Delzant}. We will use Coulon's version \cite[Theorem~7.10]{Coulon} of the relevant small cancellation theorem \cite[Theorem~3.10]{Arzhantseva-Delzant}. 

\begin{prop}\label{prop:graphical_sc} Let $\kappa<1$, $k>0$, $\eta>0$, $C>0$, $d>0$. Then there exist $\theta>0$ (depending only on $\kappa,k$), $\gamma>0$ (depending only on $\kappa,k,C$),  and $j_0>0$ such that for every integer $j\geq j_0$ we have: let $(\Theta_n)_{n\in\N}$ be a sequence of graphs of vertex degree at most $d$ with $\diam(\Theta_n)\leq C\girth(\Theta_n)$ for every $n$, and let $G$ be a non-elementary torsion-free hyperbolic group with a generating set $S$ with $|S|\leq k$, such that the spectral radius of $G$ w.r.t.\ $S$ is at most $\kappa$. Then asymptotically almost surely we have for the uniform random labelling of $\sdg nj$ by $S$:

\begin{itemize}
 \item The group $G/\sdg nj$ is non-elementary torsion-free hyperbolic.
 \item The map $G\to G/\sdg nj$ restricted to the ball of radius $\theta\cdot\girth\sdg nj$ w.r.t.\ $S$ is an isometry.
 \item For any label-preserving map $f:\sdg nj \to \Cay(G/\sdg nj,S)$ and any $x,y\in \sdg nj$ we have:
$d(f(x),f(y))\geq \gamma\cdot (d(f(x),f(y))-\eta\girth(\sdg nj)).$
\end{itemize} 
\end{prop}

The following is a statement about random walks on hyperbolic groups.

\begin{prop}[{\cite[Section~5]{Arzhantseva-Delzant}}]\label{prop:arzhantseva-delzant}
Let $\kappa<1$, $k>0$, and $\lambda>0$. Then there exist $C_0,C_1,C_2,C_3>0$ (depending only on $\kappa$ and $k$)  and $\xi_0'>0$ such that, if the sequence of finite connected graphs $(\Gamma)_{n\in\N}$ is $b$-thin (in the sense of \cite[Definition~5.3]{Arzhantseva-Delzant}) with constant $\xi_0\leq \xi_0'$, where $b=-\log(\kappa)/2$, then, for any non-elementary torsion-free hyperbolic group $H$ with spectral radius at most $\kappa$ with respect to a finite generating set $S_H$ with $|S_H|\leq k$, we have asymptotically almost surely for the uniform random labelling of $\Gamma_n$ (denoting $g_n:=\girth(\Gamma_n)$):
\begin{itemize}
 \item any label-preserving map $\phi:\widetilde {\Gamma_n}\to \Cay(H,S)$ is a $g_n$-local $(C_0,C_1\xi_0g_n)$-quasi-isometric embedding, and $g_n/2$ is greater than the corresponding threshold of \cite[Chapter~3]{Coornaert} for stability of quasi-geodesics;
 \item $\phi$ is a (global) $(2C_0,C_2\xi_0g_n)$-quasi-isometric embedding and, in particular, the minimal length in $H$ with respect to $S$ of an element represented by the label of a homotopically non-trivial closed path in $\Gamma_n$ is at least $g_n(1/(2C_0)-C_2\xi_0)$;
 \item the image of $\phi$ is $C_3\xi_0g_n$-quasi-convex;
 \item the quantity $\Delta$ of \cite[Theorem~3.10]{Arzhantseva-Delzant} is bounded above by $\lambda g_n$.
\end{itemize}
\end{prop}

Observe that the upper bound for $\Delta$ used in the proof of \cite[Lemma~5.8]{Arzhantseva-Delzant} also gives an upper bound on $\Delta'(\mathcal Q)$ in \cite[Theorem~7.10]{Coulon}.

\begin{proof}[Proof of Proposition~\ref{prop:graphical_sc}] We first claim: given $b,\xi_0>0$ there exists $j_0$ such that for any integer $j\geq j_0$, we have that $\left(\sdg nj\right)_{n\in \N}$ is $b$-thin (with constant $\xi_0$) in the sense of \cite[Definition~5.3]{Arzhantseva-Delzant}. For $\xi\in[\xi_0,1/2)$, denote by $b_n^j(\xi j\rho_n)$ the number of simple paths of length $\xi j\rho_n$ in $\sdg nj$. It is clear from \cite[Definition~5.3]{Arzhantseva-Delzant} that it is sufficient to show: for large enough $j$ there exists $K_j$ such that for every $\xi$ we have $b_n^j(\xi j\rho_n)\leq K_j\exp(b\xi_0j\rho_n)$. Now:
$$ b_n^j(\xi j\rho_n)\leq |\sdg nj|d^{1+\xi\rho_n}\leq \frac{dj}{2}d^{(\diam\Theta_n+1)+(1+\xi\rho_n)}< (d^3j/2)\exp(\log(d)(C+1/2)\rho_n).$$
Thus, our claim holds with $K_j:=(d^3j/2)$ whenever $j\geq j_0:=\left\lceil \frac{\log(d)(C+1/2)}{b\xi_0}\right\rceil$.

Let $\kappa<1$ and $k>0$, $k\in\N$, and $C_0,C_1,C_2,C_3$ the resulting constants from Proposition~\ref{prop:arzhantseva-delzant}. Choose both $\lambda>0$ and (using the resulting $\xi_0'$) $\xi_0<\min\{\xi_0',\eta/(2C_0C_1)\}$ small enough such that for any group $H$ and graph $(\Gamma_n)_{n\in\N}$ as in the assumptions of Proposition~\ref{prop:arzhantseva-delzant}, the constants obtained in the final three bullets of Proposition~\ref{prop:arzhantseva-delzant} satisfy the assumptions on $T(\mathcal Q)$ and $\Delta'(\mathcal Q)$ in \cite[Theorem~7.10]{Coulon} as $g_n\to\infty$. 

More precisely, in the notation of \cite[Theorem~7.10]{Coulon}, we can set $k:=C_0$, $\rho:=\rho_0$, which gives $\delta_2$ and $\Delta_2$. Let $\delta$ be the hyperbolicity constant of $\Cay(H,S_H)$.  Consider some $n$. Then $l=C_1\xi_0g_n$, $L=g_n/2$, and $\alpha=C_3\xi_0g_n$. We have, by Proposition~\ref{prop:arzhantseva-delzant} and the observation thereafter, that $T(\mathcal Q)\geq g_n(1/(2C_0)-C_2\xi_0)$ and $\Delta'(\mathcal Q)\leq \lambda g_n$. Thus, by choosing both $\lambda$ and subsequently $\xi_0$ small enough (only depending on $C_0,C_2,C_3$), we obtain for any large enough $g_n$, that $\delta/T(\mathcal Q)\leq \delta_2$, $\alpha/T(\mathcal Q)\leq 10\delta_2$, and $\Delta'(\mathcal Q)/T(\mathcal Q)\leq \Delta_2$ if $n$, as required. 

Choose $j_0$ such that for every integer $j\geq j_0$, $\left(\sdg nj\right)_{n\in\N}$ is $b$-thin for $b=-\log(\kappa)/2$ with our chosen constant $\xi_0$. Then \cite[Theorem~7.10]{Coulon} gives our claim. (Notice that the requirement that $\xi_0\leq\eta/(2C_0C_1)$ ensures our claim on the inequality in the third bullet.)
\end{proof}

We are ready to prove Proposition~\ref{prop:main}, which we restate for convenience.

\propmain*

\begin{proof} Let $\epsilon>0$ be obtained from Lemma~\ref{lem:many_random_words} (for our $k$). Let $\phi>0$ be the value from Lemma~\ref{lem:random_walk_avoids_balls} (for $k,\kappa$). Let $\gamma>0$ and $\theta>0$ be the values obtained from Proposition~\ref{prop:graphical_sc} (for $k,\kappa,C$). Let $\eta:=\min\{\frac 1{16}, \frac \theta2\}$ and $j_0$ the resulting value of Proposition~\ref{prop:graphical_sc}. Let $j\geq j_0$. Let $h_d^{(j)}>0$ be the lower bound for the Cheeger constants of $\sdg nj$ obtained from Lemma~\ref{lem:subdivision_cheeger_constant} (for $h,d,j$), and let $L>0$ be the constant from Lemma~\ref{lem:linear_detours_in_expanders}. 

Note that we have  $R\log|\sdg nj|\leq \girth\left(\sdg nj\right)\leq r\log|\sdg nj|$ for some $r,R>0$ depending only on $C,d,h,j$: as mentioned in the proof of Lemma~\ref{lem:linear_detours_in_expanders}, there exists $C_{h_d^{(j)}}$ (only depending on $h_d^{(j)}$) such that $\diam(\sdg nj)\leq C_{h_d^{(j)}}\log|\sdg nj|$, whence we have $\girth(\sdg nj)\leq 2\diam(\sdg nj)+1\leq 3C_{h_d^{(j)}}\log|\sdg nj|$. Furthermore, $|\sdg nj|\leq d^{\diam(\sdg nj)+1}\leq d^{2C\girth(\sdg nj)}$.

Let $\epsilon_0:=\min\{\frac{r\epsilon}{4},\frac{1}{8}\}$. By Lemma~\ref{lem:many_random_words}, a.a.s.\ every word of length at most $r\epsilon\girth\left(\sdg nj\right)$ appears on $\sdg nj$.  We apply this to the label of the concatenation of geodesics in $\Cay(G/\sdg nj)$ from $x_1$ to $m$ and from $m$ to $x_2$ for $m,x_1,x_2$ as in the statement. Thus, a.a.s. we may realize every triple $m,x_1,x_2$ as in the statement as the image of a triple of vertices $m_0,v_1,v_2$ in $\sdg nj$ under a label-preserving map $f:B_{\girth(\sdg nj)/2}(m_0)\to \Cay(G,S)$. (Notice that this ball is a tree, so the map is well-defined.) 

By construction, the labels of the unique geodesics from $m_0$ to $v_i$ are geodesic in $G$. Hence, by Lemma~\ref{lem:extendo-paths}, a.a.s. there exist geodesics $q_i$ in $\sdg nj$ starting at $v_i$ and terminating at vertices $w_i$ at distance $\bigl\lceil\frac{\girth\left(\sdg nj\right)}8\bigr\rceil$ from $m_0$ such that the $f(q_i)$ do not intersect the ball of radius $\frac{\epsilon_0\girth\left(\sdg nj\right)}{2+2\phi}$ around $m$. Observe that for each $i$, the length of $q_i$ is less than $\frac{\girth\left(\sdg nj\right)}{2}$.

Now since, by Proposition~\ref{prop:graphical_sc}, a.a.s. $\pi:G\to G/\sdg nj$ is an isometry on balls of radius $\theta\girth\left(\sdg nj\right)$, we get that the images under $\pi$ of the 
subsets
of the $f(q_i)$ contained in $B_{\theta\girth\left(\sdg nj\right)}(m)$ do not intersect the ball of radius $\frac{\epsilon_0\girth\left(\sdg nj\right)}{2+2\phi}$ around $\pi(m)$. By Proposition~\ref{prop:graphical_sc}, a.a.s. the images under $\pi$ of the remainders of the $f(q_i)$ do not intersect the ball of radius $\gamma(\theta-\eta)\girth\left(\sdg nj\right)\geq\frac{\gamma\theta}{2}\girth\left(\sdg nj\right)$ around $\pi(m)$. By Lemma~\ref{lem:linear_detours_in_expanders}, there exist a path $z$ from $w_1$ to $w_2$ of length at most $\frac{L}{R}\girth\left(\sdg nj\right)$ in $\sdg nj$ that does not intersect the ball of radius $\frac{\girth\left(\sdg nj\right)}8$ around $m_0$. By Proposition~\ref{prop:graphical_sc}, a.a.s.\ the image of $z$ in $G/\sdg nj$ (i.e. the path with the same label as $z$ that goes from $\pi(f(w_1))$ to  $\pi(f(w_2))$) does not intersect the ball of radius $\gamma(\frac{1}{8}-\eta)\girth\left(\sdg nj\right)\geq \frac{\gamma}{16}\girth \left(\sdg nj\right)$ around $\pi(m)$. Thus, our claim holds for $\epsilon_0$ as above, $\nu_0:=\min\left\{\frac{\gamma}{16},\frac{\gamma\theta}2,\frac{\epsilon_0}{2+2\phi},\right\}$, and $L_0:=\frac{L}{R}+1$.
\end{proof}

We now conclude the proof of Theorem~\ref{thm:main}, which we again restate for convenience. 

\thmmain*

Given the intermediate results we have already collected, the remainder of the proof is a variation on the limit procedure in the construction of random Gromov's monsters. For the sake of Theorem~\ref{thm:qi}, we take care to let our result also go to subsequences of $\Sigma$.

\begin{proof} Let $k:=|S|$ and let $H$ be any non-elementary torsion-free hyperbolic property~(T) quotient of $G$. Let $\kappa<1$ be the Kazhdan constant of $H$ with respect to $S$. Let $\gamma$ and $\theta$ be as in Proposition~\ref{prop:graphical_sc} (for our values of $k,\kappa,C$) and, for our given $\eta>0$, let $j_{0,1}$ be as obtained from Proposition~\ref{prop:graphical_sc} (for our $k,\kappa,\eta,C,d$). 
Let $j_{0,2}$ be as obtained from Proposition~\ref{prop:main} (for our $k,\kappa,C,d$ and $h=\inf_{n\in\N} h(\Theta_n)>0$). 
We will prove our theorem for $j_0:=\max\{j_{0,1},j_{0,2}\}$. Let $j\geq j_0$, and let $\epsilon_0,\nu_0,L_0>0$ be obtained from Proposition~\ref{prop:main}.
We inductively choose the subsequence $\Sigma=(\Sigma_1,\Sigma_2,\dots)$ of $(\sdg nj)_{n\in\N}$.

Denote by $\Sigma_0$ the empty graph. For $n\geq 1$, we (recursively) declare a labelling of $\Sigma_n$ to be ``good'' if
we have that for every $\{i_1,i_2,\dots,i_x\}\subseteq\{1,2,\dots,n-1\}$, if each of the $\Sigma_{i_y}$ is endowed with a good labelling, then both quotients $$G/(\Sigma_{i_1},\Sigma_{i_2},\dots,\Sigma_{i_x})\to G/(\Sigma_{i_1},\Sigma_{i_2},\dots,\Sigma_{i_x},\Sigma_n)$$ and $$H/(\Sigma_{i_1},\Sigma_{i_2},\dots,\Sigma_{i_x})\to H/(\Sigma_{i_1},\Sigma_{i_2},\dots,\Sigma_{i_x},\Sigma_n)$$
satisfy the conclusions of both Propositions~\ref{prop:main} and \ref{prop:graphical_sc} (with our given constants). For $n\geq 1$, given $\Sigma_0,\dots,\Sigma_{n-1}$, we choose $\Sigma_n$ such that:

\begin{itemize}
 \item[(1)] With probability at least $p^{\frac{1}{2^{n}}}$, the uniform random labelling of $\Sigma_n$ is good. 
 \item[(2)] $2\epsilon_0\leq \nu_0\lceil\epsilon_0\girth(\Sigma_n)\rceil$ and $\lceil\epsilon_0\girth(\Sigma_n)\rceil\leq 2\epsilon_0\girth(\Sigma_n)$.
 \item[(3)] $\girth(\Sigma_n)\geq \girth(\Sigma_{n-1})$ and $\theta\cdot \girth (\Sigma_{n})\geq (2\epsilon_0+\nu_0+L_0)\girth(\Sigma_{n-1})$ if $n\geq 2$.
\end{itemize}

Clearly, (2) and (3) can be achieved as $|\sdg mj|\to\infty$ as $m\to\infty$. We explain why (1) can be for $n\geq 1$: let $I:=\{i_1,i_2,\dots i_x\}\subseteq \{1,\dots,n-1\}$ and for each $i_y$, consider a good labelling of $\Sigma_{i_y}$.
Then $H_{I}:=H/(\Sigma_{i_1},\Sigma_{i_2},\dots,\Sigma_{i_x})$ is non-elementary torsion-free hyperbolic. By \cite[Proposition~7.2]{Arzhantseva-Delzant}, the Kazhdan constant $\kappa<1$ of $H$ provides an upper bound for the spectral radius w.r.t.\ $S$ of each infinite quotient of $H$, in particular for $H_{I}$. Since the spectral radius is non-decreasing with respect to quotients, $\kappa$ also bounds from above the spectral radius of $G_{I}:=G/(\Sigma_{i_1},\Sigma_{i_2},\dots,\Sigma_{i_x})$. Thus, we have that both $G_{I}$ and $H_{I}$ are non-elementary torsion-free hyperbolic groups generated by $S$ with spectral radius at most $\kappa$. Hence, both the quotient of $G_{I}$ and of $H_{I}$ by $\sdg mj$ endowed with the uniform random labelling satisfies the conclusions of both Propositions~\ref{prop:main} and \ref{prop:graphical_sc} a.a.s.\ (as $m\to\infty$).

Now, in order to choose $\Sigma_n$ among the $\sdg mj$, we have to consider all $I\subseteq\{1,2,\dots,n-1\}$ and all good labellings of $\Sigma_1,\Sigma_2,\dots,\Sigma_{n-1}$. There are finitely many possibilities. Thus, we are considering the intersection of finitely many events that occur a.a.s. Such a finite intersection occurs a.a.s., whence we can achieve (1) by choosing $m$ large enough.

By construction, for the uniform random labelling, with probability at least $p=\prod_{i=1}^\infty p^{\frac{1}{2^i}}$, each of the $\Sigma_n$ is good. 
This implies that with probability at least $p$, for each subsequence $\Omega:=(\Omega_n)_{n\in\N}$ of $\Sigma:=(\Sigma_n)_{n\in\N}$ and each $n$:

\begin{itemize}
 \item[(a)] $G/(\Omega_1,\dots,\Omega_{n})\to G/(\Omega_1,\dots,\Omega_{n+1})$ is an isometry on the ball of radius $(2\epsilon_0+\nu_0+L_0)\cdot\girth(\Omega_n)$ w.r.t.\ $S$, and
 \item[(b)] $\Cay(G/(\Omega_1,\dots,\Omega_n),S)$ contains the detours at scale $\lceil\epsilon_0\girth(\Omega_n)\rceil$ described in Proposition~\ref{prop:main}.
\end{itemize}

Notice that the detours in (b) as well as the ball they avoid are contained in the ball of radius $(2\epsilon_0+\nu_0+L_0)\cdot\girth(\Omega_n)$. Hence, by (a), they survive in the quotient $G/(\Omega_1,\dots,\Omega_{n})\to G/(\Omega_1,\dots,\Omega_{n+1})$ and in any successive quotient since the girths are non-decreasing. Thus, the detours survive in the limit $G/\Omega$.

Finally, the conclusion follows from Lemma \ref{lem:div_paths} where, in the notation of Lemma~\ref{lem:div_paths}, we set $R:=\lceil \epsilon_0\girth(\Omega_n)\rceil$, $\epsilon:=\min\{\frac{\nu_0}{2\epsilon_0},\frac14\}$, and $L:=\frac{L_0}{\epsilon_0}$. 
\end{proof}

\section{Uncountably many QI-classes from subsequences}

We now use Theorem~\ref{thm:main} to deduce that, by varying the subsequence of the expander $\sdg nj$, we obtain uncountably many quasi-isometry classes of random Gromov's monsters.

\begin{thm}\label{thm:qi} Let $G$ be a non-elementary torsion-free hyperbolic group with a finite generating set $S$, $p\in(0,1)$, and $(\Theta_n)_{n\in\N}$ a $d$-regular expander graph such that there exists $C>0$ with $\diam(\Theta_n)\leq C\girth(\Theta_n)$ for every $n$. Let $j_0$ as obtained in Theorem~\ref{thm:main}, $j\geq j_0$, and $\Sigma$ the corresponding sequence obtained in Theorem~\ref{thm:main}. Then there exists a subsequence $\Omega=(\Omega_{n})_{n\in\N}$ of $\Sigma$ such that, with probability at least $p$ for the uniform random $S$-labelling of $\Omega$,  whenever $I,J\subseteq \N$ have infinite symmetric difference, then the divergence functions of $G/(\Omega_i)_{i\in I}$ and $G/(\Omega_j)_{j\in J}$ are not equivalent.
\end{thm}

As argued in Remark~\ref{remark:constants}, we could replace ``uniform random $S$-labelling of $\Omega$'' by ``uniform random $S$-labelling of $\Sigma$'' in the statement.

\begin{proof}[Proof of Theorem~\ref{thm:qi}] 
Suppose we have constructed $\Omega_1,\dots,\Omega_{n-1}$. Consider all the good labellings (in the sense of the proof of Theorem~\ref{thm:main}) of these graphs and all resulting groups $G/(\Omega_{i_1},\Omega_{i_2},\dots,\Omega_{i_x})$ for $\{i_1,i_2,\dots,i_x\}\subseteq\{1,2,\dots,n-1\}$. Let $f_{n-1}$ be the infimum of all the divergence functions of the corresponding Cayley graphs and $\phi_{n-1}$ be the supremum. Since each of these groups is non-elementary hyperbolic by Proposition~\ref{prop:graphical_sc}, each has at least exponential divergence, see e.g. \cite[Proposition~III.H.1.6]{BrHa-metspaces}, and there exists $r_{n-1}>0$ such that for all $r\geq r_{n-1}$ we have $f_{n-1}(r)\geq r^2$. Denote $\rho_n:=\girth(\Omega_n)$. Choose $\Omega_n$ such that 
\begin{itemize}
 \item $\rho_n\geq r_{n-1}\cdot n$ and
 \item $\theta\rho_n>2\phi_{n-1}(\rho_{n-1}\cdot n)$.
\end{itemize}
Here $\theta$ is the constant coming from Proposition~\ref{prop:graphical_sc} controlling the injectivity radius ($\kappa$ the Kazhdan constant of a property~(T) quotient of $G$ as in the proof of Theorem~\ref{thm:main}). 

It follows from the construction of $\Sigma$ in the proof of Theorem~\ref{thm:main} that with probability at least $p$ for the uniform random $S$-labelling of $\Omega$, each $\Omega_n$ is good in the sequence $\Omega$. From now on, we fix a such a labelling of $\Omega$. Observe that, by definition of ``goodness'', if a labelling of $\Omega_n$ is good in the sequence $\Omega$, then it is is good in any subsequence of $\Omega$ in which $\Omega_n$ appears.

Let $I,J\subseteq\N$. As just observed, each $\Omega_i$ is good in $(\Omega_i)_{i\in I}$ and the same holds for each $\Omega_j$ in $(\Omega_j)_{j\in J}$. Suppose $|I\triangle J|=\infty$ and, without loss of generality, assume $I\setminus J$ contains an infinite set $K$. Consider $G/(\Omega_i)_{i\in I}$ and $G/(\Omega_j)_{j\in J}$. Let $f_I$ and $f_J$ be the divergence functions of the Cayley graphs of $G/(\Omega_i)_{i\in I}$ and $G/(\Omega_j)_{j\in J}$, and assume they are equivalent with comparison constant $D\geq 1$. By Theorem \ref{thm:main}, $f_I$ is bounded by a linear function along a subsequence equivalent to $(\rho_k)_{k\in K}$, which means that there exists $L$ so that for every $k\in K$ there exists $\rho'_k$ with $\frac 1L\rho_k-L\leq \rho'_k\leq L\rho_k+L$ so that $f_I(\rho'_k)\leq L\rho'_k.$ 

Let $k\in K$. Since $J$ does not contain $k$, if $k$ is large enough, then balls of radius $2\phi_{k}(\rho_{k}')$ in $\Cay(G/(\Omega_j)_{j\in J},S)$ are isometric to balls in some $G/(\Omega_{m_1},\Omega_{m_2},\dots,\Omega_{m_x})$ for $\{m_1,m_2,\dots,m_x\}\subseteq\{1,2,\dots,k-1\}$ by Proposition~\ref{prop:graphical_sc} and the second condition on $\Omega_n$ above. In particular, we have $f_J(r)\geq f_{k-1}(r)$ for every $r\leq \rho_k'$. Again for $k$ large enough, we have $\rho_k'/D \geq r_{k-1}$ and hence $f_J(\rho_k'/D)\geq \rho_k'^2/D^2$. But then $\rho_k'^2/D^2\leq f_J(\rho_k'/D)\leq Df_I(\rho_k')+\rho_k'+D\leq  DL\rho'_k+\rho_k'+D$, which cannot hold if $\rho'_k$ is large enough. Thus, $D$ is not a comparison constant, a contradiction.
\end{proof}

\begin{remark} Clearly, $r^2$ can be replaced by any  function $g(r)$ with $\frac{g(r)}{\exp(r)}\to 0$ and $\frac {g(r)}{r}\to\infty$. In fact, using the same proof, one can show that given any countable collection of subexponential functions, if the sequence of expanders is sparse enough, then the resulting group has divergence larger than all the given functions along a subsequence. See \cite{GS-smallcanc} for a similar construction. 
\end{remark}

\bibliographystyle{alpha}
\bibliography{biblio}

\end{document}